\documentclass[a4paper,10pt]{article}

\usepackage{amsmath,amssymb,amsthm,amscd}
\usepackage[mathscr]{eucal}
\usepackage{multirow}
\usepackage{hhline}
\usepackage[all]{xy}
\usepackage{pstricks}

\theoremstyle{plain}

\newcommand{\mcal}{\mathcal}

\newcommand{\mfr}{\mathfrak}
\newcommand{\mbb}{\mathbb}
\newcommand{\mrm}{\mathrm}

\newtheorem{theorem}{Theorem}[section]
\newtheorem{corollary}[theorem]{Corollary}
\newtheorem{lemma}[theorem]{Lemma}

\newtheorem{proposition}[theorem]{Proposition}
\newtheorem{conjecture}[theorem]{Conjecture}

\theoremstyle{definition}
\newtheorem{definition}[theorem]{Definition}
\newtheorem{remark}[theorem]{Remark}

\newtheorem{example}[theorem]{Example}

\newcommand{\cO}{\mathcal{O}}
\newcommand{\e}{\varepsilon}

\title{Non-existence of certain CM abelian varieties 
with prime power torsion}
\author{Yoshiyasu Ozeki\footnote{
Graduate School of Mathematics, 
Kyushu University, Fukuoka 819-0395, Japan.
\endgraf
e-mail: {\tt y-ozeki@math.kyushu-u.ac.jp}
\endgraf
Partly supported by the Grant-in-Aid for Research Activity Start-up, 
The Ministry of Education, Culture, Sports, Science and Technology, Japan.}}

\date{}
\begin{document}
\maketitle

\begin{abstract}
In this paper, 
we study a conjecture of Rasmussen and Tamagawa,
on the finiteness of the set of isomorphism classes of 
abelian varieties with constrained prime power torsion.
Our result is related with abelian  varieties 
which have complex multiplication over their fields of definition.
\end{abstract}


\section{Introduction}
Let $K$ be a finite extension of $\mbb{Q}$,
$\bar{K}$ an algebraic closure of $K$ 
and $G_K=\mrm{Gal}(\bar{K}/K)$ the absolute Galois group of $K$.
For a prime number $\ell$, 
we denote by $K(\mu_{\ell})$ the field generated by 
the $\ell$-th roots of unity.
We denote by $\mcal{A}(K,g,\ell)$
the set of $K$-isomorphism classes of 
$g$-dimensional abelian varieties $A$ over $K$ 
which satisfy the following:

\vspace{1mm}

(RT$_{\mathrm{\ell}}$)
      $K(A[\ell])$ is an $\ell$-extension of $K(\mu_{\ell})$.

\vspace{1mm}

(RT$_{\mathrm{red}}$) 
       The abelian variety $A$ has good reduction
       away from $\ell$ over $K$.

\vspace{1mm}

\noindent
It follows from the condition (RT$_{\mathrm{red}}$) and  
Faltings' result on the Shafarevich Conjecture that
$\mcal{A}(K,g,\ell)$ is a finite set.
Rasmussen and Tamagawa suggested that
such finiteness should hold if we take the union 
of these sets for $\ell$ varying over all primes.

\begin{conjecture}[\cite{RT}, Conjecture 1]
\label{RT}
The set $\mcal{A}(K,g):=\{([A],\ell)\mid 
[A]\in \mcal{A}(K,g,\ell),\ \ell:\mrm{prime\ number }\}$
is finite, that is, the set $\mcal{A}(K,g,\ell)$ is empty  
for any prime number $\ell$ large enough. 
\end{conjecture}

\noindent
This conjecture is proved only in a few case.
For example,
Conjecture \ref{RT} in the case 
where $K$ is the rational number field or certain quadratic field,
with $g=1$ 
is proved by Rasmussen and Tamagwa in \cite{RT}. 
Their proof is based on results on $K$-rational points 
of modular curves of \cite{Ma} and \cite{Mo}.
Arguments for the moduli of algebraic points on Shimura curves (\cite{Ar}, \cite{AM}) also give 
results on Conjecture \ref{RT} for QM-abelian surfaces and certain quadratic field $K$. 

In this paper, 
we prove Conjecture \ref{RT}
for abelian varieties  in $\mcal{A}(K,g,\ell)$ 
which satisfy the condition that 
representations associated with their $\ell$-adic Tate modules are abelian. 
In fact, we prove more general result as follows:
Denote by $\mcal{A}'(K,g,\ell)_{\mrm{ab}}$
the set of $K$-isomorphism classes of 
$g$-dimensional abelian varieties $A$ over $K$ 
which satisfy the following:

\vspace{1mm}

(RT$_{\mathrm{\ell}}$)$'$ For some finite extension $L$ of $K$ 
         which is unramified at all places of $K$ above $\ell$,
         $L(A[\ell])$ is an $\ell$-extension of $L(\mu_{\ell})$.

\vspace{1mm}

(RT$_{\mathrm{ab}}$)  
The representation $\rho_{A,\ell}\colon G_K\to GL(T_{\ell}(A))$
associated with the $\ell$-adic Tate module $T_{\ell}(A)$ of $A$ has an abelian image.

\vspace{1mm}

\noindent
Our main result in this paper is 

\begin{theorem}
\label{abRT}
The  set $\mcal{A}'(K,g,\ell)_{\mrm{ab}}$ is empty  
for any prime number $\ell$ large enough. 
\end{theorem}

\noindent
Keys to the proof of Theorem \ref{abRT}
are to construct a compatible system of Galois representations
which has a strong condition, Faltings' trick in his proof 
of  the Shafarevich Conjecture 
and Raynaud's criterion of semi-stable reduction. 
We hope that this study will be a first step to solve Conjecture \ref{RT}
for abelian  varieties with complex multiplication.
If an abelian variety $A$ over $K$ has complex multiplication over $K$
(in the sense of \cite{ST}, Section 4),
then it is well-known that $\rho_{A,\ell}$ is 
abelian (cf.\ {\it loc}.\ {\it cit}., Section 4, Corollary 2).
Thus we obtain 

\begin{corollary}
The set of
$K$-isomorphism classes 
of abelian varieties in $\mcal{A}(K,g,\ell)$ which have  complex multiplication over $K$
is empty for any prime $\ell$ large enough.
\end{corollary}

\noindent
We want to replace ``an abelian image'' 
in the statement of (RT$_{\mathrm{ab}}$)   
with  ``a potential abelian image''.
Assume that  
we obtain Theorem \ref{abRT} with this replacement.
Under this assumption, we see that 
Conjecture \ref{RT}
holds for CM abelian varieties, that is,
if we denote by $\mcal{A}(K,g,\ell)_{\mrm{CM}}$  the set of $K$-isomorphism classes of 
abelian varieties in $\mcal{A}(K,g,\ell)$ which have complex multiplication over $\bar{K}$,
then the set 
\[
\mcal{A}(K,g)_{\mrm{CM}}:=\{([A],\ell)\mid 
[A]\in \mcal{A}(K,g,\ell)_{\mrm{CM}},\ \ell:\mrm{prime\ number }\}
\]

\noindent
is finite.

The paper proceeds as follows.
Section 2 is devoted to a study of compatible systems.
In Section 3, we recall some facts about 
Conjecture \ref{RT}. 
Finally we prove our main theorem in Section 4.


\section{Compatible systems}

In this section, we use same notation as given in Introduction.
To find conditions that compatible systems 
should be of a simple form is important for the proof of 
Theorem \ref{abRT}.

\subsection{Basic notions}
Let $E$ be a finite extension of $\mbb{Q}$.
For a finite place $\lambda$ of $E$,
we denote by $\ell_{\lambda}$ the 
prime number below $\lambda$, 
$E_{\lambda}$ the completion of $E$ at $\lambda$ and 
$\mbb{F}_{\lambda}$ the residue field of $\lambda$.
We denote by $E_{\lambda}$ (resp.\ $K_v$) the 
completion of $E$ at a finite place $\lambda$ of $E$
(resp.\ the completion of $K$ at a finite place $v$ of $K$).
Let $S$ be a finite set of finite places of $K$ and $T$ a finite set of finite places of $E$.
Put $S_{\ell}=S\cup \{\mrm{places\ of}\ K\ \mrm{above}\ \ell \}$.
A representation $\rho\colon G_K\to GL_n(E_{\lambda})$
is said to be {\it $E$-rational with ramification set $S$} if
$\rho$ is unramified outside $S_{\ell}$ and
the characteristic polynomial $\mrm{det}(T-\rho(\mrm{Fr}_v))$
of $\mrm{Fr}_v$ has coefficients in $E$
for each finite place $v\notin S_{\ell}$ of $K$,
where $\mrm{Fr}_v$ is an arithmetic Frobenius of $v$.

Now we give definitions of compatible systems of $\lambda$-adic (resp.\ mod $\lambda$) 
representations, which mainly follows from that in \cite{Kh1} and \cite{Kh2}.  
An {\it $E$-rational  strictly 
compatible system $(\rho_{\lambda})_{\lambda}$
of $n$-dimensional $\lambda$-adic representations of $G_K$
with defect set $T$ and ramification set $S$},
consists of, for each finite place $\lambda$ of $E$
not in $T$, a continuous representation
$
\rho_{\lambda}\colon G_K\to GL_n(E_{\lambda})
$
that is 
\begin{enumerate}
\item[(i)] $\rho_{\lambda}$ is unramified outside $S_{\ell_{\lambda}}$;

\item[(ii)] for any finite place $v\notin S$ of $K$,
     there exists a monic polynomial $f_v(T)\in E[T]$ such that
     for all places $\lambda\notin T$ of $E$ which is coprime to 
     the residue characteristic of $v$, 
     the characteristic polynomial 
     $\mrm{det}(T-\rho_{\lambda}(\mrm{Fr}_v))$ of $\mrm{Fr}_v$ 
     is equal to $f_v(T)$.
\end{enumerate} 
 
\noindent
An {\it $E$-rational  strictly 
compatible system $(\bar{\rho}_{\lambda})_{\lambda}$
of $n$-dimensional mod $\lambda$ representations of $G_K$
with defect set $T$ and ramification set $S$},
consists of, for each finite place $\lambda$ of $E$
not in $T$, a continuous representation
$
\bar{\rho}_{\lambda}\colon G_K\to GL_n(\mbb{F}_{\lambda})
$
that is 
\begin{enumerate}
\item[(i)] $\bar{\rho}_{\lambda}$ is unramified outside $S_{\ell_{\lambda}}$; 

\item[(ii)] for any finite place $v\notin S$ of $K$,
     there exists a monic polynomial $f_v(T)\in E[T]$ such that
     for all places $\lambda\notin T$ of $E$ which is coprime to 
     the residue characteristic of $v$, 
     $f_v(T)$ is integral at $\lambda$ and 
     the characteristic polynomial 
     $\mrm{det}(T-\rho_{\lambda}(\mrm{Fr}_v))$ of $\mrm{Fr}_v$ 
     is the reduction of $f_v(T)$ mod $\lambda$.
\end{enumerate}

\noindent
We will often suppress the sets $S$ and $T$ from the notations. 
\begin{example}
Let $X$ be a proper smooth variety over $K$.
Let $V_{\ell}:=H^r_{\mrm{\acute{e}t}}(X_{\bar K}, \mbb{Q}_{\ell})^{\vee}$
be the dual of the $\ell$-adic \'etale cohomology group  
$H^r_{\mrm{\acute{e}t}}(X_{\bar K}, \mbb{Q}_{\ell})$ of $X$. 
Then the system $(V_{\ell})_{\ell}$ is a strict compatible system 
whose defect set is all prime numbers and 
ramification set is the set of finite places $v$ of $K$
such that $X$ has bad reduction at $v$.
This fact follows from the Weil Conjecture 
which is proved by Deligne (cf.\ \cite{De1}, \cite{De2}). 
\end{example}

It is conjectured that every $E$-rational strictly 
compatible system arises motivically.

\begin{conjecture}[\cite{Kh1}, Conjecture 1]
\label{KhC}
Any $E$-rational strictly 
compatible system of $\lambda$-adic (resp.\ mod $\lambda$) representations arises motivically.
\end{conjecture}

\noindent
In fact, this conjecture is true if representations are abelian. 

\begin{theorem}[\cite{Kh2}, Theorem 2 and Corollary 1]
\label{Kh}
An $E$-rational strictly compatible system of abelian semisimple 
$\lambda$-adic (resp. mod $\lambda$) representations of $G_K$ 
arises from $n$ Hecke characters.
\end{theorem}

For a finite place $v$ of $K$,
we denote by $G_v$ a decomposition group of $v$
and $I_v$ the inertia subgroup of $G_v$.
An {\it inertial level $\mfr{L}$ of $K$} is 
a collection $(\mfr{L}_v)_v$ of open normal 
subgroups $\mfr{L}_v$ of $I_v$ 
for each finite place $v$ of $K$ 
such that $\mfr{L}_v=I_v$ for almost all $v$.
An {\it inertial level $\mfr{L}$ of a geometric 
$\lambda$-adic representation $\rho_{\lambda}$ of
$G_K$}, where we use the notion of {\it geometric} in the sense of \cite{FM},
is 
the collection $(\mfr{L}_v(\rho_{\lambda}))_v$ 
of open normal 
subgroups $\mfr{L}_v(\rho_{\lambda})$ of  $I_v$ 
for each finite place $v$ of $K$, 
where $\mfr{L}_v(\rho_{\lambda})$
is the largest open subgroup of $I_v$ such that 
the restriction of $\rho_{\lambda}$ to 
$\mfr{L}_v(\rho_{\lambda})$  is semi-stable.
By definition,  we have $\mfr{L}_v(\rho_{\lambda})=I_v$ 
for almost all $v$.
A compatible system $(\rho_{\lambda})_{\lambda}$ of
geometric $\lambda$-adic representations of $G_K$ has 
{\it bounded inertial level} 
if there exists an inertial level 
$\mfr{L}=(\mfr{L}_v)_v$
such that $\mfr{L}_v\subset \mfr{L}_v(\rho_{\lambda})$
for all $\lambda$ and $v$.
Let $w_1,w_2,\dots ,w_n$ be integers.
A $\lambda$-adic representation $\rho_{\lambda}$ is 
{\it $E$-rational with Frobenius weights 
$w_1,w_2,\dots ,w_n$ outside $S$}
if $\rho_{\lambda}$ is $E$-rational with ramification set $S$ and 
for all finite places $v\notin S_{\ell}$ of $K$,
the complex roots of 
the characteristic polynomial 
$\mrm{det}(T-\rho(\mrm{Fr}_v))$  of $\mrm{Fr}_v$,
for a chosen embedding of $E$ into $\mbb{C}$,
have their complex absolute values 
$q^{w_1/2}_v,q^{w_2/2}_v,\dots ,q^{w_n/2}_v$
where $q_v$ is the cardinality of the residue field of $v$.
A strict compatible system 
$(\rho_{\lambda})_{\lambda}$
is said to be 
{\it $E$-rational strict compatible system 
with Frobenius weights $w_1,w_2,\dots ,w_n$} 
if each $\rho_{\lambda}$ is $E$-rational with 
Frobenius weights $w_1,w_2,\dots ,w_n$ outside 
a ramification set of $(\rho_{\lambda})_{\lambda}$.
We call $w_1,w_2,\dots ,w_n$ 
the Frobenius weights of $\rho_{\lambda}$ 
(resp.\ $(\rho_{\lambda})_{\lambda}$), and 
$\rho_{\lambda}$ (resp.\ $(\rho_{\lambda})_{\lambda}$) is said to be {\it pure} if $w_1=w_2=\cdots =w_n$. 
A compatible system $(\rho_{\lambda})_{\lambda}$ of
geometric $\lambda$-adic representations of $G_K$ has 
{\it bounded Hodge-Tate weights}
if there exist integers $a$ and $b$ with $a\le b$ such that, 
for any $\lambda$ and finite place $v$ of $K$ 
above $\ell_{\lambda}$,
all the Hodge-Tate weights of $\rho|_{G_v}$ 
viewed as a $\mbb{Q}_{\ell}$-representation are in $[a,b]$.
Finally, a compatible system $(\bar{\rho}_{\lambda})_{\lambda}$ of
mod $\lambda$ representations of $G_K$ is of
{\it bounded Artin conductor} if
there exists an ideal $\mfr{N}$ of $K$ such that,
for any $\lambda$,
the Artin conductor of 
$\bar{\rho}_{\lambda}$ divides $\mfr{N}$.

\begin{proposition}
\label{PropHecke}
$(1)$ An $E$-rational strictly compatible system 
$(\rho_{\lambda})_{\lambda}$ of
abelian semisimple $\lambda$-adic representations of $G_K$ 
has bounded inertial level and bounded Hodge-Tate weights.

\noindent
$(2)$ An $E$-rational strictly compatible system 
$(\rho_{\lambda})_{\lambda}$ of
abelian semisimple mod $\lambda$ representations of $G_K$ 
is of bounded Artin conductor.
\end{proposition}

\begin{proof}
By Theorem \ref{Kh}, such $(\rho_{\lambda})_{\lambda}$ 
arises from Hecke characters.
Hence the Proposition follows from standard properties of 
a representation arising from Hecke characters.
\end{proof}

\subsection{Structures of certain compatible systems}

Choose an algebraic closure $\bar{\mbb{F}}_{\lambda}$ of
$\mbb{F}_{\lambda}$. 
Put $\chi_{\lambda}\colon G_K
\overset{\chi_{\ell_{\lambda}}}{\longrightarrow} 
\mbb{Z}_{\ell_{\lambda}}^{\times}
\hookrightarrow E_{\lambda}^{\times}$
and 
$\bar{\chi}_{\lambda}\colon G_K
\overset{\bar{\chi}_{\ell_{\lambda}}}{\longrightarrow} 
\mbb{F}_{\ell_{\lambda}}^{\times}
\hookrightarrow \mbb{F}_{\lambda}^{\times}$, 
where $\chi_{\ell_{\lambda}}$ and 
 $\bar{\chi}_{\ell_{\lambda}}$ are 
the $\ell_{\lambda}$-adic cyclotomic character and 
the mod $\ell_{\lambda}$ cyclotomic character,
respectively.
For a representation 
$\bar{\rho}_{\lambda}\colon G_K\to GL_n(\mbb{F}_{\lambda})$
with abelian semisimplification,
Schur's lemma shows that
$(\bar{\rho}_{\lambda})^{\mrm{ss}}\otimes \bar{\mbb{F}}_{\lambda}$
conjugates to the direct sum of $n$ characters,
where the subscript ``ss'' means the semisimplification,
and we call these $n$ characters 
{\it characters associated with $\bar{\rho}_{\lambda}$}.
For a $\lambda$-adic representation $\rho_{\lambda}$,
we denote by 
$\bar{\rho}_{\lambda}$
a residual representation of $\rho_{\lambda}$ 
(for a chosen lattice).
Note that the isomorphism class of 
$(\bar{\rho}_{\lambda})^{\mrm{ss}}$ 
is independent of the choice of a lattice
by the Brauer-Nesbitt theorem.
 
\begin{theorem} 
\label{STR1} 
Let $(\rho_{\lambda})_{\lambda}$ be an 
$E$-rational strictly compatible system
of $n$-dimensional geometric semisimple $\lambda$-adic representations 
of $G_K$.
Suppose that 
there exists an infinite set $\Lambda$ 
of finite places of $E$ which satisfies the following:
\begin{enumerate}
\item[$(1)$] For any $\lambda\in \Lambda$, there exists a place $v$
of $K$ above $\ell_{\lambda}$ such that

\begin{enumerate}
\item[$(a)$] $\rho_{\lambda}$ is semi-stable at  $v$.
\item[$(b)$] there exist integers $w_1\le w_2$ which are independent 
of the choice of $\lambda\in \Lambda$ such that the Hodge-Tate weights 
of $\rho_{\lambda}|_{G_v}$ are in $[w_1,w_2]$.
\end{enumerate}

\item[$(2)$] For any $\lambda\in \Lambda$, 
$(\bar{\rho}_{\lambda})^{\mrm{ss}}$ is abelian and 
any character associated with 
$\bar{\rho}_{\lambda}$
has the form $\e\bar{\chi}^a_{\lambda}$,
where 
$a$ is an integer and 
$\e\colon G_K\to \bar{\mbb{F}}_{\lambda}^{\times}$ is a character 
unramified at all places of $K$ above $\ell_{\lambda}$.

\item[$(3)$] The Artin conductor of $(\bar{\rho}_{\lambda})^{\mrm{ss}}$
is bounded independently of the choice of $\lambda\in \Lambda$.
\end{enumerate}
Then there exist integers $m_1, m_2,\dots , m_n$ 
and a finite extension $L$ of $K$
such that,
for any $\lambda$,
the representation $\rho_{\lambda}$ is isomorphic to 
$\chi_{\lambda}^{m_1}\oplus \chi_{\lambda}^{m_2}
\oplus \cdots \oplus \chi_{\lambda}^{m_n}$ 
on $G_L$.
\end{theorem}

\begin{remark}
If Conjecture \ref{KhC} holds, 
then we can remove the conditions (1) and (3) of Theorem \ref{STR1} 
since these conditions are automatically satisfied. 
\end{remark}

\begin{proof}[Proof of Theorem \ref{STR1}]
By replacing $\Lambda$ with its infinite subset,  
we may suppose that
$\ell_{\lambda}$ does not divide the 
discriminant of $K$ and $\ell_{\lambda}>[E:\mbb{Q}]\cdot n$
for any $\lambda\in\Lambda$.
Furthermore, 
we may assume that, for any $\lambda\in \Lambda$ and 
a finite place $v$ of $K$ 
above $\ell_{\lambda}$ as in the condition (1),
the Hodge-Tate weights of
$\rho_{\lambda}|_{G_v}$ viewed as a 
$\mbb{Q}_{\ell_{\lambda}}$-representation are positive and 
bounded independently of the choice of $\lambda\in \Lambda$. 
By the condition (3),
there exists an ideal $\mfr{n}$ of $\cO_K$ 
such that, for any $\lambda\in \Lambda$,
the Artin conductor outside $\ell_{\lambda}$
of $(\bar{\rho}_{\lambda})^{\mrm{ss}}$ divides $\mfr{n}$.
If we denote by $\psi$ a character associated with  
$(\bar{\rho}_{\lambda})^{\mrm{ss}}$ for $\lambda\in \Lambda$
and decompose 
$\psi=\e\bar{\chi}^a_{\lambda}$ 
where $\e$ is as the condition (2),
then the Artin conductor outside $\ell_{\lambda}$
of $\varepsilon$ also divides $\mfr{n}$.
Hence, replacing the field $K$ with 
the strict ray class field of $K$ associated with $\mfr{n}$,
we may replace the condition
(2) with the following condition (2)$'$:
\begin{enumerate}
\item[(2)$'$] For any $\lambda\in \Lambda$, 
$(\bar{\rho}_{\lambda})^{\mrm{ss}}$ is abelian and 
any character associated with 
$\bar{\rho}_{\lambda}$
has the form $\bar{\chi}^a_{\lambda}$.
\end{enumerate}

\noindent
Now take any $\lambda\in \Lambda$.
Let $\bar{\chi}^{a_{\lambda,1}}_{\lambda}, 
\bar{\chi}^{a_{\lambda,2}}_{\lambda}, \dots, 
\bar{\chi}^{a_{\lambda,n}}_{\lambda}$
be all the characters associated with $\bar{\rho}_{\lambda}$. 
By the condition (2)$'$ and $\ell_{\lambda}>[E:\mbb{Q}]\cdot n$,
the representation $(\bar{\rho}_{\lambda})^{\mrm{ss}}$ 
conjugates\footnote{Here we use the following fact:
Let $\mbb{F}$ be a field of characteristic $\ell>0$.
Let $\rho$ and $\rho'$ be $n$-dimensional semisimple $\mbb{F}$-representation of a group $G$.
Assume that $\ell>n$. 
If $\mrm{det}(T-\rho(g))=\mrm{det}(T-\rho'(g))$ for any $g\in G$,
then $\rho$ is isomorphic to $\rho'$. 
}
to the direct some of $n$ characters
(over $\mbb{F}_{\lambda}$)
of the form $\bar{\chi}^a_{\lambda}$,
which has values in $\mbb{F}^{\times}_{\ell_{\lambda}}$. 
Hence if we regard the $\mbb{F}_{\lambda}$-representation 
$\bar{\rho}_{\lambda}$ as an 
$\mbb{F}_{\ell_{\lambda}}$-representation,
its semisimplification is of a diagonal form whose
diagonal components are the copies of 
$\bar{\chi}^{a_{\lambda,1}}_{\ell_{\lambda}}, 
\bar{\chi}^{a_{\lambda,2}}_{\ell_{\lambda}}, \dots, 
\bar{\chi}^{a_{\lambda,n}}_{\ell_{\lambda}}$
(here we note that $\ell_{\lambda}>[\mbb{F}_{\lambda}:\mbb{F}_{\ell_{\lambda}}]\cdot n$). 
Furthermore, it is a direct summand of the semisimplification 
of a residual representation of $\rho_{\lambda}$ viewed as a 
$\mbb{Q}_{\ell_{\lambda}}$-representation.
Therefore, by Caruso's result on an upper bound for
tame inertia weights 
(\cite{Ca}) and the condition (1),
there exists a constant $C>0$,
which is independent of the choice of $\lambda\in \Lambda$, 
and an integer $0\le b_{\lambda,i}\le C$
such that 
\vspace{-1mm}
\[
(\sharp)\quad b_{\lambda,i}\equiv a_{\lambda,i}\ \mrm{mod}\ \ell_{\lambda}-1
\]
\vspace{-1mm}
for any $i$
(recall that $\ell_{\lambda}$ does not divide the discriminant of $K$).
Now 
we claim that 
the set $\{b_{\lambda,1},b_{\lambda,2},\dots ,b_{\lambda,n}\}$
is independent of the choice of $\lambda\in \Lambda$ large enough. 
Denote by $S$ the ramification set of $(\rho_{\lambda})_{\lambda}$.
Take a $v_0\notin S$ and decompose
$\mrm{det}(T-\rho_{\lambda}(\mrm{Fr}_{v_0}))
=\prod^{n}_{j=1}(T-\alpha_{v_0,j})$.
By conditions (2)$'$ and $(\sharp)$,
we have the congruence 
$\prod^{n}_{j=1}(T-\alpha_{v_0,j})
\equiv 
\prod^{n}_{j=1}(T-q^{b_{\lambda,j}}_{v_0})
$
in $\bar {\mbb{F}}_{\lambda}[T]$.
If $\ell_{\lambda}$ is large enough (note that $\Lambda$ is an infinite set),
then we obtain that this congruence is in fact an equality in $E[T]$:
$
\prod^{n}_{j=1}(T-\alpha_{v_0,j})
= \prod^{n}_{j=1}(T-q^{b_{\lambda,j}}_{v_0}).
$ 
Therefore, 
the set $\{b_{\lambda,1},b_{\lambda,2},\dots ,b_{\lambda,n}\}$
is independent of the choice of $\lambda\in \Lambda$ with 
$\ell_{\lambda}$ large enough. This proves the claim.
We denote $\{b_{\lambda,1},b_{\lambda,2},\dots ,b_{\lambda,n}\}$
by $\{m_1,m_2,\dots ,m_n \}$ for such a $\lambda\in \Lambda$.
By the compatibility of $(\rho_{\lambda})_{\lambda}$, we obtain the equation
$
\mrm{det}(T-\rho_{\lambda}(\mrm{Fr}_v))
=\prod^{n}_{j=1}(T-q^{m_j}_v) 
$
for any $\lambda$ and $v\notin S_{\ell_{\lambda}}$.
Therefore, the representation $\rho_{\lambda}$ is isomorphic to 
$\chi_{\lambda}^{m_1}\oplus \chi_{\lambda}^{m_2}
\oplus \cdots \oplus \chi_{\lambda}^{m_n}$. 
By the compatibility of $(\rho_{\lambda})_{\lambda}$, this finishes the proof.
\end{proof}

\begin{corollary}
\label{ModSTR1}
Let $(\bar{\rho}_{\lambda})_{\lambda}$ 
be an $E$-rational strictly compatible system
of abelian semisimple 
mod $\lambda$ representations of $G_K$.
Suppose that, for infinitely many finite places $\lambda$ of $E$, 
any character associated with 
$\bar{\rho}_{\lambda}$
has the form $\e\bar{\chi}^a_{\lambda}$,
where $\e\colon G_K\to \bar{\mbb{F}}_{\lambda}^{\times}$ is a character 
unramified at all places of $K$ above $\ell_{\lambda}$.
Then there exist a finite extension $L$ of $K$ 
and integers $m_1, m_2, \dots, m_n$ such that, 
for any $\lambda$,
the representation $\bar{\rho}_{\lambda}$ is isomorphic to 
$\bar{\chi}_{\lambda}^{m_1}\oplus \bar{\chi}_{\lambda}^{m_2}
\oplus \cdots \oplus \bar{\chi}_{\lambda}^{m_n}$ 
on $G_L$.
\end{corollary}

\begin{proof}
By Theorem \ref{Kh}, 
we know that there exist a finite extension $E'$ of $E$
and an $E'$-rational abelian semisimple compatible system 
$({\rho}_{\lambda'})_{\lambda'}$
of $\lambda'$-adic representations of $G_K$
which arises from Hecke characters
such that $({\rho}_{\lambda'})_{\lambda'}$ 
is a lift of $(\bar{\rho}_{\lambda})_{\lambda}$,
that is, $\bar{\rho}_{\lambda'}$ is isomorphic to 
$\bar{\rho}_{\lambda}\otimes \mbb{F}_{\lambda'}$
for any $\lambda$ and any finite place $\lambda'$ of $E'$ above $\lambda$.
By standard properties of compatible systems of 
Galois representations arising from Hecke characters, 
we see that $(\rho_{\lambda'})_{\lambda'}$ satisfies 
all the assumptions (1), (2) and (3) in Theorem \ref{STR1}. 
Consequently we obtain the desired result. 
\end{proof}

\begin{corollary} 
\label{STRCor1} 
Let $(\rho_{\lambda})_{\lambda}$ be an $E$-rational strictly compatible system
of $n$-dimensional semisimple $\lambda$-adic representations 
of $G_K$. 
Suppose that

\begin{enumerate}
\item[$(\mrm{i})$] 
$(\bar{\rho}_{\lambda})^{\mrm{ss}}$ is abelian for almost all $\lambda$;

\item[$(\mrm{ii})$] 
for infinitely many $\lambda$, 
any character associated with 
$(\bar{\rho}_{\lambda})^{\mrm{ss}}$
has the form $\e\bar{\chi}^a_{\lambda}$,
where 
$\e\colon G_K\to \mbb{F}_{\lambda}^{\times}$ is a character
unramified at all places of $K$ above $\ell_{\lambda}$.
\end{enumerate}

\noindent
Then there exist integers $m_1, m_2,\dots , m_n$ 
and a finite extension $L$ of $K$
such that,
for any $\lambda$,
the representation $\rho_{\lambda}$ is isomorphic to 
$\chi_{\lambda}^{m_1}\oplus \chi_{\lambda}^{m_2}
\oplus \cdots \oplus \chi_{\lambda}^{m_n}$ 
on $G_L$.
\end{corollary}

\begin{proof}
The result follows immediately by applying Corollary \ref{ModSTR1}
to the compatible system
$((\bar{\rho}_{\lambda})^{\mrm{ss}})_{\lambda}$.
\end{proof}

Let $\lambda$ and $\lambda'$ be finite places of $E$ of 
different residual characteristics.
Let $\rho_{\lambda}$ 
be an $E$-rational $n$-dimensional semisimple 
$\lambda$-adic representations of $G_K$ with ramification set $S$.
Suppose that 
there exists 
an semisimple 
$\lambda'$-adic representation $\rho_{\lambda'}$ of $G_K$ 
such that 
\[
\mrm{det}(T-\rho_{\lambda}(\mrm{Fr}_v))=
\mrm{det}(T-\rho_{\lambda'}(\mrm{Fr}_v))
\]
for any $v\notin S_{\ell_{\lambda}}\cup S_{\ell_{\lambda'}}$.
In the spirit of Fontaine-Mazur's 
``Main Conjecture'',  
we hope\footnote{
In fact $\rho_{\lambda}$ and $\rho_{\lambda'}$
shall come from an algebraic variety $X$ and their
ramification set $S$ shall be ``bad primes'' of $X$.} that 
$\rho_{\lambda'}$ is crystalline for any finite place 
$v'$ of $K$ above $\ell_{\lambda'}$ when the residual characteristic of $\lambda'$ 
is prime to that of any place in $S$.
However to prove this hope seems not to be easy.
If $\rho_{\lambda}$ is abelian,
the hope is true by Theorem \ref{Kh}.
If we consider representations which is pure,
we can improve the statement (1) of Theorem \ref{STR1} as below.
(If the hope is true, 
it is not difficult to prove the proposition below 
without the assumption of pureness by the similar method of 
the proof of Theorem \ref{STR1}.)

\begin{proposition} 
\label{STR2} 
Let $(\rho_{\lambda})_{\lambda}$ 
be an $E$-rational strictly compatible system
of $n$-dimensional geometric semisimple 
$\lambda$-adic representations of $G_K$.
Suppose that $(\rho_{\lambda})_{\lambda}$ is pure.
Suppose that 
there exists an infinite set $\Lambda$ 
of finite places of $K$ which satisfies the following:
\begin{enumerate}
\item[$(1)$] For any $\lambda\in \Lambda$, there exists a place $v$
of $K$ above $\ell_{\lambda}$ such that

\begin{enumerate}
\item[$(a)$] there exists a constant $C>0$ which is independent 
of the choice of $\lambda\in \Lambda$ such that 
$[I_v:\mfr{L}_v(\rho_{\lambda})]<C$.
Here $\mfr{L}_v(\rho_{\lambda})$ is the inertial level of 
$\rho_{\lambda}$ at $v$ $($see Section 2.1$)$.
\item[$(b)$] there exist integers $w_1\le w_2$ which are independent 
of the choice of $\lambda\in \Lambda$ such that the Hodge-Tate weights 
of $\rho_{\lambda}|_{G_v}$ are in $[w_1,w_2]$.
\end{enumerate}

\item[$(2)$] For any $\lambda\in \Lambda$, 
$(\bar{\rho}_{\lambda})^{\mrm{ss}}$ is abelian and 
any character associated with 
$\bar{\rho}_{\lambda}$
has the form $\e\bar{\chi}^a_{\lambda}$,
where 
$\e\colon G_K\to \bar{\mbb{F}}_{\lambda}^{\times}$ is a character 
unramified at all places of $K$ above $\ell_{\lambda}$.

\item[$(3)$] For any $\lambda\in \Lambda$, 
the Artin conductor of $(\bar{\rho}_{\lambda})^{\mrm{ss}}$
is bounded independently of the choice of $\lambda\in \Lambda$.
\end{enumerate}
Then there exist an integer $m$ 
and a finite extension $L$ of $K$
such that,
for any $\lambda$,
the representation $\rho_{\lambda}$ is isomorphic to 
$(\chi^m_{\lambda})^{\oplus n}$ 
on $G_L$.
\end{proposition}

\begin{proof}

Most parts of the first paragraph of this proof 
will proceed by the similar method as the proof of
Theorem \ref{STR1} and hence we will often omit precise arguments.
First we may assume that, for any $\lambda\in \Lambda$,

\begin{enumerate}
\item[(2)$'$] 
any character associated with 
$\bar{\rho}_{\lambda}$ 
has the form $\bar{\chi}^a_{\lambda}$
\end{enumerate}
\noindent
and furthermore,
$\rho_{\lambda}|_{G_v}$ has Hodge-Tate weights in $[0,r]$ for any $\lambda$
and $v$ as in the condition (1). 
Here $r$ is a positive integer which is independent of the choice of $\lambda\in \Lambda$.
Suppose $\lambda$ is a finite place in $\Lambda$.
Let $\bar{\chi}_{\lambda}^{a_{\lambda,1}}, 
\bar{\chi}_{\lambda}^{a_{\lambda,2}}, \dots ,
\bar{\chi}_{\lambda}^{a_{\lambda,n}}$ be 
all the characters associated with $\bar{\rho}_{\lambda}$.
Taking a finite place $v$ as in the condition (1),
there exists a finite extension $L_w$ of $K_v$
such that 
$\rho_{\lambda}|_{G_{L_w}}$ is semi-stable and $[L_w:K_v]\le C$.
If we denote by $e_w$ the absolute ramification index of  $L_w$,   
then it follows $e_w\le C[K:\mbb{Q}]$, and 
Caruso's result on an upper bound for
tame inertia weights 
(\cite{Ca}) implies that 
there exists an integer $0\le b'_{\lambda,i}\le e_wr$
which satisfies
$
b'_{\lambda,i}\equiv e_wa_{\lambda,i}\ 
\mrm{mod}\ \ell_{\lambda}-1.
$ 
Consequently, 
we see that there exist integers $e>0$ and $D>0$, 
which are independent of the choice of $\lambda\in \Lambda$
and  
$
b_{\lambda,i}\equiv ea_{\lambda,i}\ 
\mrm{mod}\ \ell_{\lambda}-1
$ 
for some integer $b_{\lambda,i}\in [0,D]$.
Take any $v\notin S_{\ell_{\lambda}}$ and decompose
$\mrm{det}(T-\rho_{\lambda}(\mrm{Fr}_v))
=\prod^{n}_{j=1}(T-\alpha_{v,j})$.
Then, by the similar arguments as the proof of Theorem \ref{STR1},
we can show that 
$\prod^{n}_{j=1}(T-\alpha^e_{v,j})
=\prod^{n}_{j=1}(T-q^{b_{\lambda,j}}_v)$
if we take $\lambda\in \Lambda$
with $\ell_{\lambda}$ large enough.
Since $(\rho_{\lambda})_{\lambda}$ is pure,
we have 
\vspace{-1mm} 
\[
\prod^{n}_{j=1}(T-\alpha^e_{v,j})
=\prod^{n}_{j=1}(T-q^b_v)
\]
\vspace{-1mm} 
for some integer $b$.
It follows from the compatibility of $(\rho_{\lambda})_{\lambda}$
that the above equation holds 
for any $\lambda$ (which may not be in $\Lambda$) 
and $v\notin S_{\ell_{\lambda}}$.

In the argument below, we use
the method of the proof of Proposition 1.2 of \cite{KL}.
Fix $\lambda$ and denote it by $\lambda_0$.
Take a finite extension $K'$ of $K$ 
such that 
there exists a continuous character 
$\chi^{1/e}_{\lambda_0}\colon G_{K'} \to E^{\times}_{\lambda_0}$
which has values in the integer ring of 
$E_{\lambda_0}$ and 
$(\chi^{1/e}_{\lambda_0})^e=\chi_{\lambda_0}$. 
Replace this $K'$ with $K$.
Then we know that, for any $v\notin S_{\ell_{\lambda_0}}$,
all the roots of 
$\mrm{det}(T-\rho'_{\lambda_0}(\mrm{Fr}_v))$
are roots of unity, where 
$\rho'_{\lambda_0}$ is the twist of 
$\rho_{\lambda_0}$ by $(\chi^{1/e}_{\lambda_0})^{-b}$. 
Since there are only finitely many such roots of unity, 
there are only finitely many possibilities for the characteristic 
polynomial of $\mrm{Fr}_v$.
Hence the function which takes $g\in G_K$ to 
$\mrm{det}(T-\rho'_{\lambda_0}(g))\in E[T]$ is continuous and takes 
only finitely many values by Chebotarev's density theorem.
It follows that the set 
$\{g\in G_K\mid \mrm{det}(T-\rho'_{\lambda_0}(g))=(T-1)^n\}$
is an open subset of $G_K$, which contains the identity map of $\bar K$.
Hence there exists a finite extension $L$ of $K$ such that 
$G_L\subset \{g\in G_K\mid \mrm{det}(T-\rho'_{\lambda_0}(g))=(T-1)^n\}$.
Then we see that 
$\rho_{\lambda_0}$ is isomorphic to 
$((\chi^{1/e}_{\lambda_0})^{b})^{\oplus n}$ on $G_L$.
Since $\rho_{\lambda_0}$ is geometric, 
we know that $b/e=:m$ is an integer and we finish the proof
by the compatibility of $(\rho_{\lambda})_{\lambda}$.
\end{proof}


\section{Rasmussen-Tamagawa Conjecture}

We continue to use same notation as in the previous section.
Let $g\ge 0$ be an integer.
\begin{definition}
We denote by $\mcal{A}(K,g,\ell)$
the set of $K$-isomorphism classes 
of $g$-dimensional abelian varieties $A$ over $K$
which satisfy the following: 

\vspace{1mm}

(RT$_{\mathrm{\ell}}$)
      $K(A[\ell])$ is an $\ell$-extension of $K(\mu_{\ell})$.

\vspace{1mm}

(RT$_{\mathrm{red}}$) 
       The abelian variety $A$ has good reduction
       away from $\ell$ over $K$.
\end{definition}

\noindent
By (RT$_{\mathrm{red}}$), 
the set $\mcal{A}(K,g,\ell)$ is a finite set (Theorem 5 of \cite{Fa} and 1.\ Theorem of \cite{Za}).
Rasmussen and Tamagawa conjectured in \cite{RT} 
that for any $\ell$ large enough,
this set is in fact empty (see Conjecture \ref{RT} in Introduction).
The following results 
on the Rasmussen-Tamagawa Conjecture
are known:
\vspace{1mm}

\noindent 
(i)  (\cite{RT}, Theorem 2)
      If $K=\mbb{Q}$ and $g=1$,
      then the conjecture holds.

\vspace{1mm} 

\noindent 
(ii)  (\cite{RT}, Theorem 4)
      If $K$ is a quadratic number field other than the 
      imaginary quadratic fields of class number one 
      and $g=1$,
      then the conjecture holds. 

\vspace{1mm} 

\noindent
(iii) (\cite{Oz}, Corollary 4.5)
Let $\mcal{A}(K,g,\ell)_{\mrm{st}}$
be the set of $K$-isomorphism classes of 
abelian varieties in  
$\mcal{A}(K,g,\ell)$
with semi-stable reduction everywhere.
Then there exists an integer $C=C([K:\mbb{Q}],g)$,
depending only on $[K:\mbb{Q}]$ and $g$,
such that 
$\mcal{A}(K,g,\ell)_{\mrm{st}}$ is empty for any $\ell>C$
with $\ell \nmid d_K$.
Here $d_K$ is the discriminant of $K$.  

\noindent
(iv) (\cite{Ar}, Corollary 6.4 and  \cite{AM})
Let $K$ be a quadratic number field other than the 
imaginary quadratic fields of class number one.
Let $\mcal{A}(K,2,\ell)_{\mrm{QM}}$
be the set of $K$-isomorphism classes of 
QM-abelian surfaces in  
$\mcal{A}(K,2,\ell)$.
Then $\mcal{A}(K,2,\ell)_{\mrm{QM}}$ is empty for any $\ell$ large enough.

\vspace{1mm}

For an abelian variety $A$, 
denote by $\rho_{A,\ell}\colon G_K\to GL(T_{\ell}(A))\simeq GL_{2g}(\mbb{Z}_p)$ 
the representation 
determined by the action of $G_K$ on the $\ell$-adic Tate module $T_{\ell}(A)$ of $A$.
Consider the following properties:

\vspace{1mm}

(RT$_{\mathrm{\ell}}$)$'$ For some finite extension $L$ of $K$ 
         which is unramified at all places of $K$ above $\ell$,
         $L(A[\ell])$ is an $\ell$-extension of $L(\mu_{\ell})$.

\vspace{1mm}

(RT$_{\mathrm{ab}}$)  
The representation $\rho_{A,\ell}$
has an abelian image.

\vspace{1mm}

\noindent
It is clear that (RT$_{\mathrm{\ell}}$) implies (RT$_{\mathrm{\ell}}$)$'$.

\begin{definition}
We define sets $\mcal{A}(K,g,\ell)_{\mrm{ab}}$ and 
$\mcal{A}'(K,g,\ell)_{\mrm{ab}}$ of isomorphism 
classes of $g$-dimensional abelian varieties $A$ over $K$ as follows:

\noindent
(1) $[A]\in \mcal{A}(K,g,\ell)_{\mrm{ab}}$ if and only if 
$A$ satisfies (RT$_{\mathrm{\ell}}$), (RT$_{\mathrm{red}}$) and (RT$_{\mathrm{ab}}$).

\noindent
(2) $[A]\in \mcal{A}'(K,g,\ell)_{\mrm{ab}}$ if and only if 
$A$ satisfies (RT$_{\mathrm{\ell}}$)$'$ and (RT$_{\mathrm{ab}}$).
\end{definition}

\noindent
Clearly, we have $\mcal{A}(K,g,\ell)\supset
\mcal{A}(K,g,\ell)_{\mrm{ab}}\subset 
\mcal{A}'(K,g,\ell)_{\mrm{ab}}$.
Note that 
abelian varieties in $\mcal{A}'(K,g,\ell)_{\mrm{ab}}$
are not forced the reduction hypothesis 
(RT$_{\mathrm{red}}$).
Hence  
$\mcal{A}'(K,g,\ell)_{\mrm{ab}}$ may be 
infinite (but the author does not know 
an example such that $\mcal{A}'(K,g,\ell)_{\mrm{ab}}$ is infinite).

\section{Proof of Theorem\ \ref{abRT}}

In this section, we use same notation as
in the previous section.
First we study the structure of $A[\ell]$ for 
an abelian variety $A$ in
$\mcal{A}'(K,g,\ell)_{\mrm{ab}}$.
Let $A$ be any $g$-dimensional abelian variety over $K$.
We denote by $\bar{\rho}_{A,\ell}\colon G_K\to GL(A[\ell])\simeq GL_{2g}(\mbb{F}_p)$
the representation 
determined by the action of $G_K$ on $A[\ell]$.
Consider the following properties:

\vspace{1mm} 

(RT$_{\mrm{mod}}$) $(\bar{\rho}_{A,\ell})^{\mrm{ss}}$
              conjugates to the direct sum of $n$ characters        
              which are of the form $\bar{\chi}^a_{\ell}$.

\vspace{1mm} 

(RT$_{\mrm{mod}}$)$'$ $(\bar{\rho}_{A,\ell})^{\mrm{ss}}$ is abelian and
                 characters associated with 
                 $\bar{\rho}_{A,\ell}$
                 are of the form $\e \bar{\chi}^a_{\ell}$,
                 where $\e\colon G_K\to \bar{\mbb{F}}^{\times}_{\ell}$
                 is a continuous character which is 
                 unramified at all places above $\ell$.

\vspace{1mm} 

\noindent
The condition (RT$_{\ell}$) is equivalent 
to the condition (RT$_{\mrm{mod}}$)
by the Lemma below.
Hence the $K$-isomorphism class $[A]$ of $g$-dimensional 
abelian variety $A$ over $K$ is in $\mcal{A}(K,g,\ell)$ if and only if 
$A$ satisfies (RT$_{\mrm{mod}}$) and (RT$_{\mrm{red}}$).

\begin{lemma}
\label{structure}
Let $A$ be a $g$-dimensional abelian variety over $K$.

\begin{enumerate}
\item[$(1)$] The abelian variety $A$ satisfies 
    $(\mrm{RT}_{\ell})$ if and  
    only if $A$ satisfies $(\mrm{RT}_{\mrm{mod}})$.

\item[$(2)$] Suppose that the abelian variety $A$ satisfies
       $(\mrm{RT}_{\mrm{ab}})$.
      Then $A$ satisfies  $(\mrm{RT}_{\ell})'$ if and  
      only if $A$ satisfies $(\mrm{RT}_{\mrm{mod}})'$.
\end{enumerate}
\end{lemma}

\begin{proof}
The assertion (1) follows from the arguments of the proof of 
Lemma 3 in \cite{RT}
and thus we omit the proof.
Suppose that an abelian variety $A$ satisfies the condition
$(\mrm{RT_{ab}})$ and denote by $\psi_1,\dots ,\psi_{2g}$  
characters associated with $\bar{\rho}_{A,\ell}$.
If $A$ satisfies (RT$_{\mrm{mod}}$)$'$, then we have   
$\psi_i=\e_i\bar{\chi}^{a_i}_{\ell}$ for some integer $a_i$
where $\e_i\colon G_K\to \bar{\mbb{F}}^{\times}_{\ell}$
is a continuous character which is 
unramified at all places of $K$ above $\ell$.
Let $L$ be the composition field of all fields $\bar{K}^{\ker{\e_i}}$
for all $i$. Then $L$ is unramified at all places of $K$ above $\ell$. 
Since each $\psi_i|_{G_{L(\mu_{\ell})}}$ is trivial, we obtain
(RT$_{\mathrm{\ell}}$)$'$.
Conversely, suppose that 
(RT$_{\mathrm{\ell}}$)$'$ holds and take 
a field $L$ as in the statement of (RT$_{\mathrm{\ell}}$)$'$.
By (1),
we know that each $\psi_i|_{G_L}$ 
is equal to  $\bar{\chi}^{a_i}_{\ell}$
for some integer $a_i$.
Hence $\e_i:=\psi_i\cdot \bar{\chi}^{-a_i}_{\ell}\colon G_K\to \bar{\mbb{F}}^{\times}_{\ell}$ is unramified at 
all places above $\ell$ and this implies (RT$_{\mrm{mod}}$)$'$.
\end{proof}

We recall the following two propositions. 

\begin{proposition}[Faltings]
\label{FaCor}
Fix an integer $w$.
The set of isomorphism classes of semisimple 
$n$-dimensional $\ell$-adic representations 
$G_K\to GL_n(\mbb{Q}_{\ell})$ 
which are $\mbb{Q}$-integral 
with Frobenius weights $\le w$ outside $S$, is finite. 
\end{proposition}

\begin{proof}
The Proposition follows from the proof of Theorem 5 in \cite{Fa}.
See also \cite{La}, Chapter VIII, Section 5, Theorem 11.
\end{proof}

\begin{proposition}
[Raynaud's criterion of semi-stable reduction, \cite{Gr}, Proposition
 4.7]
\label{Racri}
Suppose $A$ is an abelian variety over a field $F$ with a 
discrete valuation $v$,
$n$ is a positive integer not divisible by the residue characteristic,
and the points of $A[n]$ are defined over an extension of $F$
which is unramified over $v$.
In particular, 
if $A$ is an arbitrary abelian variety over a number field $K$,
then $A$ has semi-stable reduction everywhere
over $K(A[12])=K(A[3],A[4])$.
\end{proposition}

For an integer $g> 0$,
put
\[
D_g:=\sharp GL_{2g}(\mbb{Z}/3\mbb{Z})\cdot 
\sharp GL_{2g}(\mbb{Z}/4\mbb{Z}). 
\]
If $\rho\colon G_K\to GL_{2g}(\mbb{Q}_{\ell})$ is an abelian representation,
then, for any integer $k$, 
we denote by $\rho^{k}$ the representation 
$G_K\to GL_{2g}(\mbb{Q}_{\ell})$ 
which is defined by $\rho^k(s):=(\rho(s))^k$
for any $s\in G_K$.
With this notation,
we obtain the following lemma
which plays an important role in the proof of Theorem \ref{abRT}
to construct a good compatible system.

\begin{lemma}
\label{Dsq}
Let $g>0$ be an integer and $\ell_0$ a prime number.
Let $\mcal{A}_{\ell_0}$ be the set 
of isomorphism classes of representations $\rho\colon G_K\to GL_{2g}(\mbb{Q}_{\ell_0})$ which are
isomorphic to $\rho^{D_g}_{A,\ell_0}$ 
for some  $g$-dimensional abelian variety $A$ over $K$ 
such that $K(A[\ell_0^{\infty}])$ is an abelian extension of $K$.  
Then $\mcal{A}_{\ell_0}$ is finite.
\end{lemma}

\begin{proof}
If $A$ is an abelian variety over $K$ such that 
$K(A[\ell_0^{\infty}])$ is an abelian extension of $K$,
then $A$ has potential good reduction everywhere. 
Putting $L:=K(A[12])$, 
such an abelian variety $A$ has good reduction everywhere over $L$
by Proposition \ref{Racri}. 
Since $[L:K]$ divides $D_g$,
the representation  $\rho^{D_g}_{A,\ell_0}$ 
is unramified outside $\ell_0$
for any  $g$-dimensional abelian variety $A$ over $K$ such that 
$K(A[\ell_0^{\infty}])$ is an abelian extension of $K$.  
Take any finite place $v$ of $K$ not above $\ell_0$.
Let $v_L$ be a finite place of $L$ above $v$ and 
denote by $f$ the extension degree of $\mbb{F}_{v_L}$ over $\mbb{F}_{v}$,
where $\mbb{F}_{v_L}$ and $\mbb{F}_{v}$ are residue fields of $v_L$ and $v$, respectively.
Noting that $L$ is a Galois extension of $K$ and $A$ has good reduction everywhere over $L$, 
we see that $D_g/f$ is an integer
and obtain the equation
\[
\mrm{det}(T-\rho^{D_g}_{A,\ell_0}(\mrm{Fr}_v))=
\mrm{det}(T-(\rho_{A,\ell_0}(\mrm{Fr}_{v_L}))^{D_g/f}).
\]
Since $A$ has good reduction everywhere over $L$,
the polynomial
$\mrm{det}(T-\rho_{A,\ell_0}(\mrm{Fr}_{v_L}))$ 
has rational integer coefficients
and hence so is $\mrm{det}(T-(\rho_{A,\ell_0}(\mrm{Fr}_{v_L}))^{D_g/f})$.
Consequently, 
the representation $\rho^{D_g}_{A,\ell_0}$ is $\mbb{Q}$-integral with
Frobenius weight $D_g/2$ outside 
the set of finite places of $K$ above $\ell_0$.
Therefore, by Proposition \ref{FaCor},
it is enough to prove that
the representation $\rho^{D_g}_{A,\ell_0}$ is semisimple.
Note that it has already known that $\rho_{A,\ell_0}$ is semisimple (Theorem 3 of \cite{Fa}).
Since $\rho_{A,\ell_0}$ is abelian and geometric in the sense of  \cite{FM}, 
the representation $\rho_{A,\ell_0}$  is locally algebraic in the sense of \cite{Se}
(see also Proposition of Section 6 in \cite{FM}).
Therefore, by (MT 1) of \cite{Ri},
there exists a modulus of definition $\mathfrak{m}$ and 
an algebraic homomorphism 
$\phi\colon S_{\mathfrak{m}}\to GL_{2g}$ over $\mbb{Q}$ such that 
the $\ell_0$-representation induced by $\phi$ is isomorphic to $\rho_{A,\ell_0}$
Here, the definition of the commutative algebraic group $S_{\mathfrak{m}}$
over $\mbb{Q}$ is given in Chapter II of \cite{Se}. 
Note that any $\ell_0$-adic representation coming from 
an algebraic morphism $S_{\mathfrak{m}}\to GL_{2g}$
is automatically semisimple.
Since  $\rho^{D_g}_{A,\ell_0}$
comes from the composition 
$S_{\mathfrak{m}}\overset{D_g}{\to} 
S_{\mathfrak{m}}\overset{\phi}{\to} GL_{2g}$
where $S_{\mathfrak{m}}\overset{D_g}{\to} 
S_{\mathfrak{m}}$ is the multiplication by $D_g$ map,
we obtain the fact that $\rho^{D_g}_{A,\ell_0}$ is semisimple. 
\end{proof}

\begin{proof}[Proof of Theorem \ref{abRT}]
First we note that, 
if an abelian variety $A$ over $K$ satisfies (RT$_{\mathrm{ab}}$),
then $\rho_{A,\ell'}$ is abelian for any prime number $\ell'$
(cf. \cite{Se}, Chapter III, Section 2.3, Corollary 1). 
Fix a prime number $\ell_0$ and denote by $\mcal{A}_{\ell_0}$ the set as in Lemma \ref{Dsq}.
Assume that there exist infinitely many prime numbers $\ell$ 
such that $\mcal{A}'(K,g,\ell)_{\mrm{ab}}$ is not empty.
For every such $\ell$, 
we obtain the $\ell_0$-adic representation $\rho^{D_g}_{A,\ell_0}$ 
which is in the set $\mcal{A}_{\ell_0}$,
where $A$ is an abelian variety
whose isomorphism class is in the set $\mcal{A}'(K,g,\ell)_{\mrm{ab}}$.
By Lemma \ref{Dsq}, we see that
there exists a representation $\rho_{\ell_0}$ in $\mcal{A}_{\ell_0}$ such that 
for infinitely many $\ell$ and 
$[A]\in \mcal{A}'(K,g,\ell)_{\mrm{ab}}$,
$\rho^{D_g}_{A,\ell_0}$ is isomorphic to $\rho_{\ell_0}$. 
In particular, we know the fact that
the representation  $\rho_{\ell_0}$ extends to a $\mbb{Q}$-integral 
strict compatible system $(\rho_{\ell})_{\ell}$
of $2g$-dimensional abelian semisimple $\ell$-adic representations 
of $G_K$.
Furthermore, 
for infinitely many prime numbers $\ell$, 
the characters associated with a residual representation $\bar{\rho}_{\ell}$ of
$\rho_{\ell}$ 
are of the form $\e\bar{\chi}^{a}_{\ell}$
by Lemma \ref{structure}, 
where $\e\colon G_K\to \bar{\mbb{F}}^{\times}_{\ell}$ is a continuous character
which is unramified at all places of $K$ above $\ell$.
Applying Theorem \ref{STR1} or Corollary \ref{STRCor1},
we see that there exist integers $m_1,\dots ,m_{2g }$ and 
a finite extension $L$ of $K$ such that  
$\rho_{\ell_0}$ is isomorphic to 
$\chi_{\ell_0}^{m_1}\oplus \chi_{\ell_0}^{m_2}
\oplus \cdots \oplus \chi_{\ell_0}^{m_{2g}}$ 
on $G_L$.
In particular, for some prime number $\ell$ and 
$[A] \in \mcal{A}'(K,g,\ell)_{\mrm{ab}}$,
$\rho^{D_g}_{A,\ell_0}$ is isomorphic to
$\chi_{\ell_0}^{m_1}\oplus \chi_{\ell_0}^{m_2}
\oplus \cdots \oplus \chi_{\ell_0}^{m_{2g}}$ 
on $G_L$.
Therefore, looking at the eigenvalues of images of 
a Frobenius element (at some place) of  
 $\rho^{D_g}_{A,\ell_0}$ and 
$\chi_{\ell_0}^{m_1}\oplus \chi_{\ell_0}^{m_2}
\oplus \cdots \oplus \chi_{\ell_0}^{m_{2g}}$, 
we know that
$D_g/2=m_1=m_2=\cdots =m_{2g}$.
Since $\rho^{D_g}_{A,\ell_0}$ has Hodge-Tate weights $0$ and $D_g$
at a place of $L$ above $\ell_0$,
this is a contradiction.
\end{proof}


\begin{thebibliography}{10}

\bibitem[Ar]{Ar}
Keisuke Arai, 
\emph{Galois images and modular curves}, 
Algebraic Number Theory and Related topics 2010 -- a volume in RIMS-Bessatsu Series --, in press.

\bibitem[AM]{AM}
Keisuke Arai and Fumiyuki Momose, 
\emph{Algebraic points on Shimura curves of $\Gamma_0(p)$-type}, 
preprint.


\bibitem[Ca]{Ca}
Xavier Caruso, 
\emph{Repr\'esentations semi-stables de torsion dans le case 
$er<p-1$}, 
J. Reine Angew. Math. {\bf 594} (2006),
35--92.

\bibitem[De1]{De1}
Pierre Deligne,
\emph{La conjecture de Weil I}, 
Publ. Math. Inst. Hautes \'Etudes Sci. {\bf 43} (1974),
273--308.

\bibitem[De2]{De2}
Pierre Deligne,
\emph{La conjecture de Weil II}, 
Publ. Math. Inst. Hautes \'Etudes Sci. {\bf 52} (1980),
137--252.

\bibitem[Fa]{Fa}
Gerd Faltings, 
\emph{Finiteness theorems for abelian varieties 
over number fields},
Arithmetic Geometry, Chapter II,
9--27.

\bibitem[FM]{FM}
Jean-Mark Fontaine and Barry Mazur, 
\emph{Geometric Galois Representations},
Elliptic curves, modular
forms, and Fermat's last theorem (Hong Kong, 1993), 
Internat. Press, Cambridge, MA (1995),
41--78.

\bibitem[Gr]{Gr}
Alexander Grothendieck, ``Mod\`eles de N\'eron et monodromie, in Groupes de monodromie
en g\'eometrie alg\'ebrique, SGA 7'', 
\textit{Lecture Notes in Mathematics} {\bf 288} (1972), 
313--523.


\bibitem[Kh1]{Kh1}
Chandrashekhar Khare, 
{\it Compatible system of mod $p$ Galois representations
and Hecke characters},
Math. Res. Lett. {\bf 10} (2003),
71--83.

\bibitem[Kh2]{Kh2}
Chandrashekhar Khare, 
{\it Reciprocity law for compatible systems of 
abelian ${\rm mod}\,p$ Galois representations},
Canad. J. Math. {\bf 57} (2005),  no. 6,
1215--1223.

\bibitem[KL]{KL}
Mark Kisin and Gus Lehrer,
{\it Eigenvalues of Frobenius and Hodge numbers},
Pure Appl. Math. Q. {\bf 2} (2006),
497--518.

\bibitem[La]{La}
Serge Lang,
{\it Algebraic number theory},
Second edition, Graduate Texts in Mathematics {\bf 110},
Springer-Verlag 
(1994).

\bibitem[Ma]{Ma}
Barry Mazur,
\emph{Rational isogenies of prime degree},
Invent. Math. {\bf 44} (1978),
129--162.


\bibitem[Mo]{Mo}
Fumiyuki Momose,
\emph{Isogenies of prime degree over number fields},
Compos. Math. {\bf 97} (1995),
329--348.


\bibitem[Oz]{Oz}
Yoshiyasu Ozeki,
\emph{Non-existence of certain Galois representations
with a uniform tame inertia weight},
Int. Math. Res. Not. IMRN {\bf 2011} (2011), no. 11,
2377--2395.

\bibitem[Ri]{Ri}
Kenneth A. Ribet
\emph{Galois Action on Division Points of Abelian Varieties with Real Multiplications},
Amer. J. Math. {\bf 98} (1976), no. 3,
751--804.


\bibitem[RT]{RT}
Christopher Rasmussen and Akio Tamagawa, 
\emph{A finiteness conjecture on abelian varieties with
constrained prime power torsion}, Math. Res. Lett. {\bf 15} (2008),
1223--1231.

\bibitem[Se]{Se}
Jean-Pierre Serre,
\emph{Abelian $l$-adic representations and elliptic curves},
second ed., 
Advanced Book Classics, 
Addison-Wesley Publishing Company Advanced Book Program,
Redwood City, CA, 1989,
With the collaboration of Willem Kuyk and John Labute.

\bibitem[ST]{ST}
Jean-Pierre Serre and John Tate,
\emph{Good reduction of abelian varieties},
Ann. of Math. (2) {\bf 8} (1986),
492--517.

\bibitem[Za]{Za}
Yuri G. Zarhin,
\emph{A finiteness theorem for unpolarized abelian varieties over
number fields with prescribed places of bad reduction},
Invent. Math. {\bf 79} (1985), 
309--321.
\end{thebibliography}
\end{document}